\documentclass[12pt,a4paper,leqno,twoside,report]{amsart}

\usepackage{latexsym}
\usepackage{verbatim}
\usepackage{amsmath}
\usepackage{amsfonts}
\usepackage{amssymb}
\usepackage[utf8]{inputenc}

\theoremstyle{plain}
\newtheorem{thm}{Theorem}
\newtheorem{cor}[thm]{Corollary}
\newtheorem{lem}[thm]{Lemma}
\newtheorem{prop}[thm]{Proposition}

\theoremstyle{definition}
\newtheorem{Def}[thm]{Definition}

\newtheorem{ex}[thm]{Example}

\def\N{\mathbb N}
\def\Z{\mathbb{Z}}
\def\Z2{\mathbb Z_2}

\def\C{\mathcal{C}}
\def\S{\mathcal{S}}
\def\T{\mathcal{T}}

\def\d{\delta}

\def\df{\partial_\tau ^{F}}
\def\dg{\partial_\tau ^{G}}

\def\D{\Delta}

\def\t{\tau}
\def\ot{\otimes}
\def\oth{\mathbin{\hat{\otimes}}}
\def\dif{\partial_{\otimes} ^F}

\def\dig{\partial_{\otimes} ^G}

\def\r{\rho}

\def\cb{C_* (B)}
\def\ce{C_* (E)}
\def\cg{C_* (G)}
\def\cf{C_* (F)}

\def\ep{\varepsilon}
\def\di{\partial}
\def\co{\circ}
\def\EML{\mathrm{EML}}
\def\AW{\mathrm{AW}}
\def\SH{\mathrm{SH}}
\def\lra{\Leftrightarrow}
\def\ra{\Rightarrow}
\def\la{\Leftarrow}
\input xy
\xyoption{all}

\begin{document}
\title{Effective chain complexes for twisted products}
\author{Marek Filakovsk\'{y}}

\address{Department of Mathematics and Statistics\\ Masaryk University\\
Kotl\'{a}\v{r}sk\'{a} 2, 611 37 Brno, Czech Republic}
\email{xfilakov@math.muni.cz}

\begin{abstract}
 In the paper weak sufficient conditions for the reduction of the chain complex of a twisted cartesian product $F \times_{\t}B$ to a chain complex of free finitely generated abelian groups are found.
\end{abstract}
\maketitle

\section{Introduction}

When making algorithmic calculations with simplicial sets in algebraic topology on the level of chain complexes, 
it is often useful to replace the chain complex $C_* (X)$ associated to a simplicial set $X$ with another chain complex $EC_*$ 
where all the groups $EC_n$ are finitely generated free abelian. Such chain complexes are called \emph{effective}. This replacement is usually 
obtained using a reduction or a strong equivallence.

It is therefore natural to ask if standard topological constructions with simplicial sets are reflected by our replacements. 
For example by the theorem of Eilenberg and Zilber we know that given simplicial sets $X, Y$ and their effective chain 
complexes $EC_* (X), EC_* (Y)$, the simplicial set $X \times Y$ has an effective chain complex ${EC_* (X) \ot EC_* (Y)}.$

Let $F\to E\to B$ be a Kan fibration of simplicial sets. By \cite{may} we may think of the total 
space $E$ as $E = F \times_\t B$, i.e. a twisted cartesian product. We want to find an effective chain complex of the total 
space $E$ from the knowledge of effective chain complexes of $F$ and $B$ and the twisting operator $\tau$.

In \cite{serg} (Theorem 132) the solution of this problem was given in the case when the space $B$ is $1$--reduced, 
which means that the $1$--skeleton of $B$ is a point. However, this condition seems to be too restrictive and not necessary.
For example if we aim to generalize the results in the paper \cite{cadek} and construct an equivariant version of the Postnikov 
tower one cannot assume the base spaces are even 0--reduced (see \cite{aslep}).
In Theorem \ref{thm1} and Corollary \ref{pro2} we give weaker conditions under which an effective chain complex for the twisted cartesian product can be found. 
Our approach is based on the results by Shih as presented in \cite{shih} and on the approach from the paper \cite{stasheff}.

\section{Basic notions}

Let $(C_*, \di), (D_*, \di)$ be chain complexes. The triple of maps $\r = (f, g, h)$ 
where $f :C_* \to D_*, g:D_* \to C_*$ are chain homomorphisms and $h:C_* \to C_{* + 1}$ is a chain homotopy such that 
\[
\begin{array}{rlcr}
gf - id_{C_*}&= \di h + h \di,&{ }& hh = 0, \\
gh&= 0, &{ }&fh = 0
\end{array}
\]
is called a reduction. The chain complex $D_*$ is said to be a reduct of $C_*$. We will denote this by $C_*\ra D_*$.

This definition of reduction coincides with the one given in \cite{serg}, Definition 42 or \cite{stasheff}, 2.1.
It is easy to observe that a composition of reductions is a reduction.
We say, there is a \emph{strong equivalence} between chain complexes ${C}_*$ and ${C'}_*$
if there exists a chain complex $D_*$ together with two reductions $\r_1 = (f_1, g_1, h_1):D_* \ra {C}_*$ and 
$\r_2 = (f_2, g_2, h_2):D_* \ra {C'}_*$. We denote this by ${C}_* \la {D}_* \ra {C'}_*$ or ${C}_* \lra {C'}_*$.
The following lemma shows that strong equivalences are in some sense composable.
\begin{lem}[\cite{serg}, Proposition 125] \label{comp}
Let $A_* \lra B_*$ and $B_* \lra C_*$ be strong equivalnces of chain complexes. Then there is a strong equivalence
${A_* \lra C_*}$.
\end{lem}
We omit the proof, it can be found in \cite{serg}. We will make use of the following "tensor product" of reductions.
\begin{lem}\label{tens}
 Let  $\r_C = (f_C, g_C, h_C):C_* \ra {C'}_*$ and $\r_D = (f_D, g_D, h_D):D_* \ra {D'}_* $ be reductions.
Then there is a reduction $$\r_{C \ot D} = (f_{C \ot D}, g_{C \ot D}, h_{C \ot D}): C_* \ot D_* \ra {C'}_* \ot {D'}_*. $$
\end{lem}
\begin{proof}
 The new reduction is defined by
$f_{C \ot D} = f_C \ot f_D$, $g_{C \ot D} = g_{C} \ot g_{D}$, 
$h_{C \ot D} = h_C \ot id_D  +  g_C f_C \ot h_D$,  or $h_{C \ot D} = h_C \ot g_D f_D +  id_C \ot h_D.$
\end{proof}

Further we will deal only with chain complexes which are formed by free abelian groups.
For any simplicial set $X$ there is a canonically associated chain complex $C_{*}(X)$ where the group $C_n (X)$ is freely generated by nondegenerate ${n}$--simplices of
$X$ and the boundary homomorphism $\di_n$ is induced by face maps in $X_n$ as follows 
\[
\di_n = \displaystyle\sum\limits_{i=0}^{n} (-1)^{n} d_i. 
\]
Let $(C_* , \di)$ be a chain complex. A collection of maps $\d_n : C_n \to C_{n-1}$ is called a \emph{perturbation} if $(\di_n +\d_n)^2 = 0$ for all $n\in \mathbb{N}$.
We will now introduce the Basic Perturbation Lemma. It is a powerful tool that enables us to construct
new reductions.
\begin{prop}[Basic Perturbation Lemma, \cite{shih}] \label{bpl}
Let $\rho = (f, g, h): (C_*, \di)\ra (D_*, \di' )$ be a reduction and let $\d$  be a perturbation of the differential $\di$. 
If for every $c\in C_n$ there exists an $\alpha \in \N$ such that
$(h \d )^{\alpha} (c)= 0$, then there is a reduction  $$\rho' = (f', g', h'): (C_*, \di + \d)\ra (D_*, \di' + \d ' ), $$ 
where $\d '$ is a perturbation of the differential $\di'$.
\end{prop}
\begin{proof}
The maps involved in the reduction $\rho '$ are given explicitely as follows:
\[ \begin{array}{ll}

f' &= f\co (1 + (\d h) + (\d h)^2 + (\d h)^3 + \ldots), \\
g' &= (1 +(h \d) + (h \d)^2 + (h \d)^3 + \ldots) \co  g, \\
{h}' &= (1 +(h \d) + (h \d)^2 + (h \d)^3 + \ldots) \co  h =  h\co (1 + (\d h) + (\d h)^2 + (\d h)^3 + \ldots), \\
\d '  &= f \co \d \co (1 + (h \d) + (h \d)^2 + (h \d)^3 + \ldots) \co  g.
\end{array} \]

The proof can be found in \cite{serg}, Theorem 50.
\end{proof}
On the other hand, if we add a perturbation to the differential of the other chain complex, we easily get the following result:
\begin{lem}[Easy Perturbation Lemma] \label{epl}
Let $\rho = (f, g, h): (C_*, \di)\ra (D_*, \di' )$ be a reduction and let $\d'$  be a perturbation of the differential $\di'$. Then there is a reducion
$\rho = (f, g, h): (C_*, \di + \d)\ra (D_*, \di' + \d')$, where
$\d = g\d'f$.
\end{lem}

The difficulty with the BPL consists in the fact that it is sometimes difficult to verify the nilpotency assumption.
Instead of looking for a description of $(h\d)$ we can find a filtration and check how how the perturbation $\d$ changes the 
filtration.

\begin{Def}
Let $B$ and $F$ be simplicial sets and let $E = F\times B$. Let $(y, b) \in E$. We may assume $b = s_{*} b' \in B$, 
where $s_{*}$ is a composition of degeneracy operators and $b'$ is nondegenerate.
The filtration degree of $(y, b)$ is the dimension of $b'$.
The filtration degree of an nonzero element $y\ot b \in \cf \ot \cb$ is the dimension of $b$.
\end{Def}

\section{twisting cochains}
The twisted cartesian product (TCP) is defined as follows:
\begin{Def} \label{dfn1}
Let $B,F$ be simplicial sets and $G$ a simplicial group with a right action $\cdot: F\times G \to F$. 
A function $\t_n: B_n \to G_{n-1}, n \geq 1$, is said to be a \emph{twisting operator}, if it satisfies the following properties:
\begin{enumerate}
 \item $d_0 \t(b) = \tau(d_1 b) \cdot \tau(d_0 b)^{-1}$,
 \item $d_i \t(b) = \tau(d_{i+1} b) , i > 0 $, 
 \item $s_i \t(b) = \tau(s_{i+1} b) , i \geq 0$, 
 \item $\t(s_0 b) = e_m$, if $b\in B_{m+1}$ where $e_m$ is the unit element of $G_m$.
\end{enumerate}
The twisted cartesian product with the base $B$, the fiber $F$ and the group $G$ is a simplicial set denoted $E$ or $F \times_{\tau} B$ where 
$E_n = F_n \times B_n$ has the following face and degeneracy operators:
\begin{enumerate}
 \item $d_0 (y,b) = (d_0 (y) \cdot \t(b), d_0 (b))$, 
 \item $d_i (y,b) = (d_i (y), d_i (b)), \, i>0$,
 \item $s_i (y,b) = (s_i (y), s_i (b)), \, i \geq 0$.
\end{enumerate}
\end{Def}
The face and degeneracy operations on $E$ naturally define a differential $\di_\t$ on the chain complex $C_* (E)$. Note that $\di_\t (y_0, b_0) = 0$ for 
$( y_0, b_0) \in F_0 \times_\tau B_0$ since $d_0 (y_0, b_0)$ is not defined.

We now introduce the following notation: If $X$ is a simplicial set 
and $x\in X_n$ we put ${\tilde{d}}^{n-i} x = d_{i+1}\cdots d_n x$ and ${\tilde{d}}^{0} x = x$.
Given $(x,y) \in (X\times Y)_n$ we define the Alexander-Whitney operator:
\[
\AW (x,y) = \displaystyle\sum\limits_{i=0}^{n} \tilde{d}^{n-i} x \ot {d_0}^{i} y.
\]
For a non--twisted product $F\times B$, there exists a reduction 
$$(\AW, \EML, \SH): (C(F\times B), \di) \ra (\cf \ot \cb, \dif),$$
known as the Eilenberg--Zilber reduction. For the full description of the reduction see \cite{eml}.

The only difference between the chain complexes ${(C_* (F \times_{\tau} B), \di_\t)}$ and ${(C_* (F\times B), \di)}$ is in their differentials and it is easy to see that
$$
\di_\t = \di  + (d_0 (y) \cdot \t(b),  d_0 (b)) - (d_0 (y),  d_0 (b)).
$$
So the differential $\di_\t$ of $\ce$ is just the $\di$ with the added perturbation 
\begin{equation*}
\d_\t = (d_0 (y) \cdot \t(b),  d_0 (b)) - (d_0 (y),  d_0 (b)).
\end{equation*}
\begin{prop}[\cite{serg}, Theorem 131]\label{TWEZ}
Let $F\times_\t B$ be a twisted product of simplicial sets. Then the Basic Perturbation Lemma can be applied to the reduction data
$(\AW, \EML, \SH): C_* (F\times B, \di) \ra (\cf \ot \cb, \dif)$ to obtain the reduction 
\[
(f,g,h): (C_* (F \times_\t B), \di_\t) \ra (\cf \oth \cb, \df),
\]
where $\cf \oth \cb$ is just $\cf \ot \cb$ with a new differential $\df$. 
\end{prop}
According to \cite{shih}, the perturbation $\df - \dif$ can be seen as a cap product with so called twisting cochain, which is induced by $\t$. We will now give definitions of those notions.

Let $t:C_{*} (B) \to C_{*-1}(G)$ be a sequence of abelian group homomorphisms $t_{n}: C(B)_n \to C(G)_{n-1}$. 
We define a few operators that will be used within the construction:
$$
D = \AW \co C_* (\D): C_* (B) \to \cb \ot \cb,
$$
where $C_*(\D)$ is induced by the diagonal map $\D: B \to B \times B$ and
$$\sigma = C(\cdot) \co \EML : \cf \ot \cg \to \cf.$$
Finally, we define the cap product $(t\cap\,):C_* (F)\ot C_* (B) \to C_* (F)\ot C_* (B)$ as a composition 
$$({\sigma} \ot 1)(1\ot t \ot 1)(1 \ot D).$$
Observe, that the cap product is a homomorphism of graded abelian groups and not of chain complexes. 
We say that $t$ is a \emph{twisting cochain} if 
$$
(\dif + (t\cap))^2 = \dif (t\cap) + (t\cap) \dif  + (t\cap)(t\cap) = 0.
$$

We saw  that the twisting operator $\tau$ induces via the BPL a new differential $\df$ on the 
chain complex $\cf \ot \cb$.
Then the same twisting operator $\tau$ (this time seen as a part of the twisted cartesian product $G\times_\t B$) also induces a differential $\dg$ on the 
chain complex $\cg \ot \cb$.

According to \cite{shih}, the twisting operator induces a \emph{twisting cochain} $t:C_{*}(B) \to C_{*-1}(G)$ as follows:
$$
t_n: C_n(B) \xrightarrow{e_0 \ot 1}  C_0 (G) \ot C_n (B) \xrightarrow{ \lambda_0 (\dg - \dig)} C_{n-1} (G) \ot C_0 (B)  \xrightarrow{p} C_{n-1} (G), 
$$
where $e_0$ is the unit element of $G_0$, $\lambda_0$ is a projection on the summand $C_{n-1} (G) \ot C_0 (B)$ of the sum 
$$
(C_* (G) \ot C_* (B))_{n-1} = \displaystyle\sum\limits_{i=0}^{n-1}  C_{n-1-i}(G) \ot C_i (B)
$$
and $p(x \ot b) = (\ep b)x$ where the map $\ep:C_0 (B) \to \mathbb{Z}$ is the augmentation.

The following proposition was formulated and proved by Shih in \cite{shih} and describes the relation between $t$ and $\df$.
\begin{prop}[\cite{shih}, Theorem 2]\label{pshi}
Let $F \times_\t B$ be a TCP and let $t$ be the twisting cochain induced by the differential $\dg$ of the chain complex $\cg \oth \cb$. 
Then $\df - \dif = t\cap.$
\end{prop}
Let $E = F \times_\t B$ be a twisted product of simplicial sets, $t$ be a twisting cochain induced by the differential $\dg$ 
on the chain complex $\cg \oth \cb$ and $b\in B_n, y \in F_k$. Then using the definition of $\AW$ and $t \cap$ together with the 
fact that $t(\tilde{d}^{n} b) = 0$ we obtain the following formula:
\begin{equation}\label{eq}
t\cap(y \ot b) = (-1)^{k}  \sigma (y \ot t(\tilde{d}^{n-1} b)) \ot d_0 b + \displaystyle\sum\limits_{i=2}^{n} (-1)^{k} \sigma (y \ot t(\tilde{d}^{n-i} b)) \ot {d_0}^{i} b.
\end{equation}

Using this formula we can summarize some properties of $t \cap$. 
\begin{cor}[\cite{stasheff}, Lemma 3.4]\label{lemmo}
Let $E = F \times_\t B$ be a twisted product of simplicial sets and let $t$ be a twisting cochain induced by the differential $\dg$ 
on the chain complex $\cg \oth \cb$. Then the following holds:
\begin{enumerate}
 \item The perturbation $(t\cap): {\cf \ot \cb} \to {\cf \ot \cb}$ lowers the filtration degree by at least one.
 \item If for all $b \in B_1, t(b) = 0$, then the perturbation $(t\cap)$ lowers the filtration degree by at least two.
\end{enumerate}
\end{cor}
\begin{proof}
The first part is clear by the formula (\ref{eq}). If $t(\tilde{d}^{n-1} b) = 0$ for all $b \in B_n$, then 
\[
t\cap(y \ot b) = \displaystyle\sum\limits_{i=2}^{n} (-1)^{k} \sigma (y \ot t(\tilde{d}^{n-i} b)) \ot {d_0}^{i} b. 
\]
which proves the second part.
\end{proof}
\section{Effective chain complex for twisted product}
We would like to find an answer to the following problem: Let $B$ and $F$ be simplicial sets, $G$ a simplicial group, $E = F\times_\t B$ a TCP, and $\r_B :C_{*}(B) \ra EC_{*}(B), \r_F :C_{*}(F) \ra EC_{*}(F)$ be 
reductions to effective chain complexes. Is there a reduction of the chain complex $C_{*}(E)$ to an effective chain complex which can be 
obtained from $\r_B, \r_F$ and $\tau$ by the application of the Basic Perturbation Lemma?

Our aim is to find an answer using the composition of given reductions. Having reductions $\r_B, \r_F$ we can by the Lemma \ref{tens} construct the reduction 
$$
\r_{F \ot B}: C_{*}(F) \ot C_{*}(B) \ra EC_{*}(F) \ot EC_{*}(B).
$$
We know that the chain homotopy $h_{F \ot B}$ from the reduction $\r_{F \ot B}$ raises the filtration degree by at most $1$. 
This follows from the fact that $h_{B}$ raises the filtration degree by at most 1 and the proof of Lemma \ref{tens}. 
We can use the BPL to construct a reduction ${\r_{E} = (f, g, h):C_{*}(E) \ra C_{*}(F) \oth C_{*}(B)}$.
From Corollary \ref{lemmo}, the perturbation operator $\df -\dif  = t\cap$ lowers the filtration degree by at least one.
If the composition $h_{F\ot B} \co (\df-\dif)$ decreased the filtration, it would be nilpotent and hence we could use the BPL 
on the reduction data $\r_{F \ot B}$ and the perturbation $\df-\dif$ to get a reduction 
$$
\r_{t}: C_{*}(F) \oth C_{*}(B) \ra EC_{*}(F) \oth EC_{*}(B)
$$
to an effective chain complex $EC_{*}(F) \oth EC_{*}(B)$ which is $EC_{*}(F) \ot EC_{*}(B)$ with a new differential obtained from the BPL.
However, in full generality $h_{F\ot B} \co (\df-\dif) = h_{F\ot B} \co (t\cap)$ preserves the filtration degree.

From (\ref{eq}) we see that in the composition $h_{F \ot B} \co (t\cap) (y\ot b)$, where $b\in B_n$, there is only one element with the filtration degree $n$, namely
\begin{equation}\label{eq2}
g_F f_F \sigma (y \ot t(\tilde{d}^{n-1} b)) \ot h_B d_0 b 
\end{equation}
and the degree $n$ element in $(h_{F \ot B} \co (t\cap))^i (y\ot b)$ is $y_i \ot b_i$ where
\[
\begin{array}{ll}
b_0 = b, & b_{i+1} = h_B d_0 b_i = {(h_B d_0)}^{i} b, \\
y_0 = y,& y_{i+1} = g_F f_F \sigma(y_i \ot t(\tilde{d}^{n-1} b_i )).
\end{array}
\]
Now we can establish conditions for $(h_{F \ot B} \co (t\cap))^i$ to decrease the filtration and prove the following theorem.
\begin{thm}\label{thm1}
Let $B$ and $F$ be simplicial sets, $G$ a simplicial group with an action on $F$, $E = F\times_\t B$ a TCP,  
and $\r_B :C_{*}(B) \ra EC_{*}(B), \r_F :C_{*}(F) \ra EC_{*}(F)$ be reductions to effective chain complexes. 

If for all $n\in \mathbb{N}, b \in B_n, y\ot b \in C_* (F) \ot C_* (B)$, there exists  $i\in \mathbb{N}$ such that $(h_B d_0)^i b = 0$ (thus $h_B d_0$ is nilpotent) or $y_i = 0$, then
there is a reduction from the chain complex $C_{*}(E)$ to an effective chain complex $EC_{*}(F) \oth EC_{*}(B)$ 
which can be obtained from $\r_B, \r_F$ and $\tau$ by the application of the Basic Perturbation Lemma.
\end{thm}

\begin{cor}\label{mycoro}
If $G$ is $0$--reduced or $\r_B$ is trivial (i.e. $f_B = g_B = id, h_B = 0$), $C_* (E)$ can be reduced to an effective chain complex using the BPL.
\end{cor}
\begin{proof}
If the reduction $\r_B$ is trivial, then the chain homotopy $h_B$ is trivial, so $h_B = 0$ and hence $b_1 = h_B d_0 = 0$.
To prove the case when $G$ is $0$--reduced
we compute $t(b)$ where $b\in B_1$. According to the definition we get 
\[
t(b) = t_1(b) = p \lambda_0 (\dg - \dig)(e_0 \ot b).
\]
From the Basic Perturbation Lemma we get
\[
\begin{array}{ll}
(\dg - \dig)(e_0 \ot b) = \AW (1 + \d_\t \SH + (\d_\t \SH)^2 + (\d_\t \SH)^3 + \ldots )\d_\t \EML (e_0 \ot b) \\
= \AW (1 + \d_\t \SH + (\d_\t \SH)^2 + (\d_\t \SH)^3 + \ldots )\d_\t (s_0 (e_0), b) \\
= \AW (1 + \d_\t \SH + (\d_\t \SH)^2 + (\d_\t \SH)^3 + \ldots )(d_0 s_0 (e_0) \cdot \t(b),  d_0 (b)) - (d_0 s_0 (e_0),  d_0 (b)) \\
= \AW (1 + \d_\t \SH + (\d_\t \SH)^2 + (\d_\t \SH)^3 + \ldots )(\t(b),  d_0 (b)) - (e_0,  d_0 (b)).
\end{array}
\]
As the operator $\SH = 0$ on $(F \times B)_0$ the only nonzero term of ${(\dg - \dig)(e_0 \ot b)}$ is
\[
\AW (\t(b),  d_0 (b)) - (e_0,  d_0 (b)) =  (\t(b) \ot d_0 (b)) - (e_0 \ot d_0 (b)),
\]
so we have 
\[
t(b) = t_1(b) = p \lambda_0 (\t(b) \ot d_0 (b)) - (e_0 \ot d_0 (b)) = \t(b) - e_0.
\]
If the group $G$ is $0$--reduced, $\t(b) = e_0$ as $e_0$ is the only element in $G_0$ and we have
$t(b) = 0$ for $b \in B_1$. That is why ${y_1 = g_F f_F \sigma(y \ot t(\tilde{d}^{n-1} b )) = 0}$ and we can apply the previous theorem.
\end{proof}

Now we turn to strong equivalences.

\begin{cor}\label{pro2}
Let $B$ and $F$ be simplicial sets, $G$ a simplicial group, ${E = F\times_\t B}$ a TCP, and ${C_{*}(B) \lra EC_{*}(B)}, {C_{*}(F) \lra EC_{*}(F)}$  
strong equivalences with effective chain complexes. 
If $G$ is $0$--reduced or $\r_B$ is trivial (i.e. $EC_{*}(B) = C_{*}(B)$ and all reductions are trivial) then $C_* (F \times_\t B)$ is strongly equivalent to an effective chain complex $EC_{*}(F) \oth EC_{*}(B)$ which can be obtained from the strong equivalences for $C_{*}(B)$ and $C_{*}(F)$ representing $C_{*}(E)$ and an effective chain complex using the Basic and Easy Perturbation Lemmas.
\end{cor}
\begin{proof}
By Proposition \ref{TWEZ} we have a reduction $C_* (F \times_\t B) \ra C_{*}(F) \oth C_{*}(B)$. Since strong 
equivalences are composable, it remains to show that there is a strong 
equivalence $C_{*}(F) \oth C_{*}(B) \lra EC_{*}(F) \oth EC_{*}(B)$.

Having strong equivalences $C_* (B) \la D_* (B) \ra EC_{*}(B)$ and  $C_* (F) \la D_* (F) \ra EC_{*}(F)$ then by Lemma \ref{tens} there is a strong equivalence  
\[
{C_* (F) \ot C_* (B) \la D_* (F)  \ot D_* (B) \ra EC_{*}(F) \ot EC_{*}(B)}\]
consisting of two reductions:
\[
\begin{array}{rl}
\r_1 &= (f_1, g_1, h_1): C_* (F) \ot C_* (B) \la D_* (F)  \ot D_* (B),  \\
\r_2 &= (f_2, g_2, h_2): D_* (F)  \ot D_* (B) \ra EC_* (F) \ot EC_* (B).
\end{array}
\]
Given the perturbation $(t\cap)$ on the chain complex 
$C_* (F) \ot C_* (B)$, we can use the Easy Perturbation Lemma on the reduction   $\r_1 = (f_1,g_1,h_1): C_* (F) \ot C_* (B) \la D_* (F)  \ot D_* (B)$ to get a new reduction 
\[
\r_1 = (f_1,g_1,h_1): C_* (F) \oth C_* (B) \la D_* (F)  \oth D_* (B),
\]
where we introduce a perturbation $g_1 (t \cap ) f_1 $ to the differential of the chain complex $ D_* (F)  \ot D_* (B) $ and the reduction data remains unchanged.
If the nilpotency condition of the composition $(g_1 (t \cap ) f_1) \co h_2$ was satisfied, we could apply the Basic Perturbation Lemma on 
the reduction data ${\r_2 = (f_2, g_2, h_2): D_* (F)  \ot D_* (B) \ra EC_* (F) \ot EC_* (B)}$ to obtain a reduction 
\[
\r_2 ' : D_* (F)  \oth D_* (B) \ra EC_* (F) \oth EC_* (B). 
\]

If $G$ is $0$--reduced, then the filtration degree of the perturbation $g_1 (t \cap ) f_1 $ is $-2$ by Corollaries \ref{lemmo} and \ref{mycoro} 
and as the the filtration degree of $h_2$ is $+1$, the nilpotency condition is satisfied. 
For $\r_B$ trivial, $h_2$ is $0$ and the nilpotency condition is trivially satisfied.

The reductions $\r_1, \r_2 '$ therefore establish a strong equivalence 
\[
{C_* (F) \oth C_* (B) \lra EC_{*}(F) \oth EC_* (B)} 
\]
and, as the strong equivalences are composable, we get $C_* (F \times_\t B) \lra EC_{*}(F) \oth EC_* (B)$.
\end{proof}

\section{Vector fields}
We will now deal with the case in which we have more information about the reduction $\r_B : C_* (B) \ra EC_{*}(B)$. In particular,  $\r_B$ is obtained via a
\emph{discrete vector field}. 
A discrete vector field $V$ on a simplicial set $X$ is a set of ordered pairs $(\sigma, \t)$, where $\sigma, \t $ are nondegenerate 
simplices of $X$, $\sigma = d_i \t$ for exactly one index $i$ and for every two distinct pairs $(\sigma, \t)$, $(\sigma ', \t ')$ we 
have $\sigma'\neq \sigma, \t' \neq \t, \sigma'\neq \t$ and $\t' \neq \sigma$.
By writing $V(\sigma) = \t$, we mean  $(\sigma, \t)\in V$.
Given a dicrete vector field $V$, the nondegenerate simplices of $X$ are divided into three subsets $\S, \T, \C$ as follows:
\begin{itemize}
 \item $\S$ is the set of \emph{source} simplices i.e. the simplices $\sigma$ such that $(\sigma, \t)\in V$,
 \item $\T$ is the set of \emph{target} simplices i.e. the simplices $\t$ such that $(\sigma, \t)\in V$,
 \item $\C$ is the set of \emph{critical} simplices i.e the remaining ones, not occuring in any edge of $V$.
\end{itemize}

A discrete vector field $V$ on a simplicial set $X$ induces a reduction $\r_X = (h_X, f_X, g_X) :C_{*}(X) \ra D_{*}(X)$ (see \cite{vector},\cite{matous}).
We will concentrate on the properties of the induced chain homotopy $h_X$. 
It turns out that $h_X(\sigma)\in \mathbb{Z} \T$ for any $\sigma$ and more importantly $h_X(\sigma) = 0$ whenever $\sigma \in \C \cup \T$.

\begin{Def}
Let $X$ be a simplicial set. For any nondegenerate simplex $\sigma \in X_n$ we will consider the following condition:
\begin{equation*}\tag{$*$}
d_0 \sigma \in \S \quad \text{implies} \quad \sigma \in \S 
\end{equation*}

We say that a discrete vector field $V$ on a simplicial set satisfies $(*)$ if all nondegenerate simplices of 
$X$ satisfy $(*)$. 
\end{Def}

\begin{cor}\label{mycorr}
Let $B$ and $F$ be simplicial sets, $G$ a simplicial group, $E = F\times_\t B$ a TCP and $\r_B :C_{*}(B) \ra EC_{*}(B), \r_F :C_{*}(F) \ra EC_{*}(F)$ be reductions to effective chain complexes. If the reduction $\r_B$ is induced by a vector field satisfying $(*)$, then
there exists a reduction from the chain complex $C_{*}(E)$ to 
an effective chain complex which can be obtained from $\r_B, \r_F$ and $\tau$.
\end{cor}
\begin{proof}
We show that ${(h_B d_0)^2 = 0}$. Under our conditions for any $b\in B_n$ we have
$b_1 = h_B (d_0 b) \in \mathbb{Z} \T$. As $h_B$ satisfies $(*)$, we see that $d_0 b_1 \in \mathbb{Z}(\C \cup \T)$ and 
consequently, $b_2 = h_B d_0 b_1 = 0$ and we can apply Theorem \ref{thm1}.
\end{proof}

\begin{ex}
An example of a vector field satisfying $(*)$ is so called Eilenberg--Mac\-Lane vector field. 
Let us have $X = K( \mathbb{Z} ,1)$. In the standart model which is infinite (see \cite{may}), 
the simplex $\sigma \in X_n$ can be 
represented as an $n$--tuple $[a_1| \ldots|a_n]$, where $a_1, \ldots, a_n \in \mathbb{Z}$  (see \cite{matous}, page $5$). The face operators are 
$d_0 \sigma = [a_2| \ldots|a_n]$, $d_n \sigma = [a_1| \ldots|a_{n-1}]$,
$d_i \sigma = [a_1| \ldots |a_{i-1}|a_{i}+a_{i+1}|a_{i+2}|\ldots|a_n]$, where $1<i<n$.

For any $\sigma = [a_1| \ldots|a_n]\in X_n$, where $a_n \neq 1$, 
we define the Eilenberg-MacLane vector field $V_\EML$ in the following way:
\[
V_\EML(\sigma) = 
\left\{ 
\begin{array}{ll}
[a_1 |\ldots|a_{n-1}|a_n -1 |1] &\text{for} \quad a_n > 1, \\ 
{[a_1|\ldots|a_{n-1}|1]} &\text{for} \quad a_n < 0.
\end{array}
\right
.\]
Now we can classify the simplices:
\begin{itemize}
 \item $\sigma \in \S$ has the form $[a_1| \ldots|a_n]$, where $a_n \neq 1$ and $n>0$.
 \item $\sigma \in \T$ has the form $[a_1 | \ldots|a_{n-1}|1]$, where $n>1$.
 \item $\sigma \in \C$ is $[]$ and $[1]$.
\end{itemize}
It is easy to check that the vector field $V_\EML$ satisfies $(*)$. Note that Corollary \ref{mycorr} implies that for any 
$E =  F \times_\t K( \mathbb{Z} ,1)$ there is a reduction $C_{*}(E) \ra EC_{*}(E)$ to an effective chain complex if there is a reduction 
$C_{*}(F) \ra EC_{*}(F)$ to an effective chain complex.
\end{ex}


\begin{thebibliography}{9}

\bibitem{cadek}
  M. \v{C}adek, M. Kr\v{c}\'{a}l, J. Matou\v{s}ek, F. Sergeraert, L. Vok\v{r}\'{\i}nek, U. Wagner:
  \emph{Computing all maps into a sphere}.
  Proceedings of the Twenty-Third Annual ACM-SIAM Symposium on Discrete Algorithms (2012).

\bibitem{aslep}
  M. \v{C}adek, M. Kr\v{c}\'{a}l, J. Matou\v{s}ek, F. Sergeraert, L. Vok\v{r}\'{\i}nek, U. Wagner:
  \emph{Algorithmic solvability of lifting extension problem}, 
  in preparation.

\bibitem{eml}
S.Eilenberg, S. MacLane:
\emph{On the group $H(\Pi, n)$}. I.
Ann. of Math. 58(1953), 55-106.

\bibitem{matous}
  M. Kr\v{c}\'{a}l, J. Matou\v{s}ek, F. Sergeraert:
  \emph{Polynomial-time homology for simplicial Eilenberg--MacLane spaces}.
  arXiv:1201.6222v1 (2012).

\bibitem{stasheff}
  L. Lambe, J. Stasheff:
 \emph{Applications of Perturbation Theory to Iterated Fibrations}.
  manuscripta math. 58 (1987), 363-376.
 
\bibitem{may}
  J. P. May:
  \emph{Simplicial Objects in Algebraic Topology}.
  The University of Chicago Press, Chicago,
  (1967), edition printed in 1992.

\bibitem{vector}
  A. Romero, F. Sergeraert:
  \emph{Discrete Vector Fields and Fundamental Algebraic Topology}.
  arXiv:1005.5685v1 (2010).

\bibitem{rub}
  J. Rubio:
  \emph{Homologie effective des espaces de lacets itérés: un logiciel}.
    Thèse, Institut Fourier, Grenoble, 1991

\bibitem{serg}
  J. Rubio, F. Sergeraert:
  \emph{Constructive Homological Algebra and Applications}, Genova Summer School 2006, arXiv:1208.3816v2.

\bibitem{shih}
 Shih   Weishu:
 \emph{Homologie des espaces fibrés}, Publications
Mathématiques  de l'IHES, vol.13 (1962), 5-87.


\end{thebibliography}
\end{document}